\documentclass[11pt, letterpaper]{amsart}
\usepackage{hyperref, color,amssymb}

\newtheorem{thm}{Theorem}[section]

\newtheorem{lem}[thm]{Lemma}
\newtheorem{prop}[thm]{Proposition}
\theoremstyle{definition}
\newtheorem{defn}[thm]{Definition}

\theoremstyle{remark}
\newtheorem{rem}[thm]{Remark}
\newtheorem{exa}[thm]{Example}
\numberwithin{equation}{section}
\newcommand{\abs}[1]{\left\vert#1\right\vert}
\newcommand{\set}[1]{\left\{#1\right\}}
\newcommand{\Real}{\mathbb R}

\newcommand{\Natural}{\mathbb N}
\newcommand{\nin}{n \in \Natural}
\newcommand{\kin}{k \in \Natural}

\newcommand{\A}{\mathcal{A}}
\newcommand{\B}{\mathcal{B}}
\newcommand{\such}{{\ | \ }}
\newcommand{\limn}{\lim_{n \to \infty}}
\newcommand{\limsupn}{\limsup_{n \to \infty}}

\newcommand{\limk}{\lim_{k \to \infty}}
\newcommand{\dfn}{\, := \,}

\newcommand{\prob}{\mathbb{P}}
\newcommand{\oprob}{\overline{\prob}}

\newcommand{\plim}{\Lb_+^0 \text{-} \lim}

\newcommand{\plimn}{\plim_{n \to \infty}}
\newcommand{\plimk}{\plim_{k \to \infty}}

\newcommand{\conv}{\mathsf{conv}}
\newcommand{\oconv}{\overline{\mathsf{conv}}}

\newcommand{\qprob}{\mathbb{Q}}
\newcommand{\expec}{\mathbb{E}}
\newcommand{\expecp}{\expec_\prob}
\newcommand{\expecq}{\expec_\qprob}

\newcommand{\Lb}{\mathbb{L}}
\newcommand{\lz}{\Lb^0}
\newcommand{\li}{\Lb^\infty}
\newcommand{\lip}{\li_{+}}
\newcommand{\lzp}{\lz_{+}}
\newcommand{\lp}{\Lb_+}
\newcommand{\lonep}{\Lb_+^1}
\newcommand{\lonepq}{\lonep(\qprob)}

\newcommand{\F}{\mathcal{F}}

\newcommand{\ud}{\mathrm d}

\newcommand{\pbdd}{$\lzp$-bounded}

\newcommand{\simplex}{\triangle^{\Natural}}

\newcommand{\hhat}{\widehat{h}}

\newcommand{\pare}[1]{\left(#1\right)}
\newcommand{\bra}[1]{\left[#1\right]}
\newcommand{\dbra}[1]{[\kern-0.15em[ #1 ]\kern-0.15em]}
\newcommand{\dbraco}[1]{[\kern-0.15em[ #1 [\kern-0.15em[}
\newcommand{\dbraoc}[1]{]\kern-0.15em] #1 ]\kern-0.15em]}
\newcommand{\C}{\mathcal{C}}

\newcommand{\Sl}{\mathcal{S}}
\newcommand{\Slo}{\Sl^\circ}

\newcommand{\indic}{\mathbb{I}}

\newcommand{\absco}{{<\kern-0.53em<}}

\newcommand{\hti}{{\widetilde{h}}}

\newcommand{\PP}{{\mathbb P}}

\newcommand{\FF}{{\mathcal F}}

\newcommand{\el}{{\mathbb L}} %l-pees
\newcommand{\lzer}{\el^0}
\newcommand{\lone}{\el^1}

\newcommand{\lpee}{\el^p}

\newcommand{\linf}{\el^{\infty}}
\renewcommand{\conv}{\mathrm{conv}}
\renewcommand{\oconv}{\overline{\conv}}
\begin{document}

\title[Forward-convex convergence of sequences in
$\lz_+$]{Forward-convex convergence in probability of sequences of nonnegative random variables}%
\author{Constantinos Kardaras}%
\address{Constantinos Kardaras, Mathematics and Statistics Department,
  Boston University, 111 Cummington Street, Boston, MA 02215, USA.}%
\email{kardaras@bu.edu}%

\author{Gordan {\v{Z}}itkovi{\'{c}}}%
\address{Gordan {\v{Z}}itkovi{\'{c}}, Department of Mathematics,
  University of Texas at Austin, 1 University Station, C1200, Austin,
  TX 78712, USA}%
\email{gordanz@math.utexas.edu}%

\thanks{The authors would like to thank Freddy Delbaen and Ted Odell for valuable help, numerous conversations and shared expertize. Both authors acknowledge partial support by the National Science Foundation; the first under award number DMS-0908461, and the second under award number DMS-0706947. Any opinions, findings and conclusions or recommendations expressed in this material are those of the authors and do not necessarily reflect those of the National Science Foundation.}%
%\keywords{}%

\date{\today}%
% \dedicatory{}%
% \commby{}%
% ----------------------------------------------------------------
\begin{abstract}
  For a sequence $(f_n)_{\nin}$ of nonnegative random variables, we provide simple necessary
  and sufficient conditions for convergence in probability of each sequence $(h_n)_{\nin}$ with $h_n\in\conv(\set{f_n,f_{n+1},\dots})$ for all $\nin$ to
  the same limit. These conditions correspond to an essentially measure-free version of the notion of uniform integrability.
\end{abstract}

\maketitle

% ----------------------------------------------------------------

\section*{Introduction}

A growing body of work in applied probability in general, and in the
field of mathematical finance in particular, has singled out 
$\lzer$, the Fr\'echet space of a.s.-equivalence classes of random
variables topologized by the convergence in probability, as
especially important (see, e.g., \cite{MR1768009, FilKupVog09, KraSch99, Zit09}). Reasons for this are multiple, but if a
single commonality is to be found, it would have to be the fact that
$\lzer$ is measure-free. More precisely, the $\lzer$-spaces built over
the same measure space with different probabilities will coincide as
long as the probabilities are equivalent. The desirability and necessity of
the measure-free property in mathematical finance stems from the
central tenet of replication (popularized by the work of Black,
Scholes, Merton and others) which finds its mathematical expression in
the theory of stochastic integration. Since replication amounts to
complete removal of risk, the probability measure under which a
financial system is modeled should not matter, modulo its negligible
sets. On the other hand, given that general stochastic integration
does not admit a canonical pathwise definition, we are left with
$\lzer$ as the only proper setting for the theory. The only other
measure-free member of the $(\lpee)_{p \in [0, \infty]}$ family, namely $\linf$, turns out
to be inadequately small for a large number of  modeling tasks.

It is important to note that the interplay between $\lzer$, the
measure-free property, and stochastic integration, reaches farther into the
history  than the relatively recent progress in
mathematical finance. The seminal work of Stricker (\cite{Str75}) on
the semimartingale property under absolutely-continuous changes of
measures and the celebrated result of Dellacherie and Bichteller
(\cite{Bic79,Bic81,Del80}) on the theory of $\lzer$-integrators are
but two early examples. Even before that, results related to the
 measure-free structure of $\lzer$, but without relation to
stochastic integration, have been published (see, e.g.,
\cite{BuhLoz73,Nik71}).

While $\lzer$ seems to fit the modeling requirements perfectly, there
is a steep price that needs to be paid for its use: a large
number of classical functional-analytic tools which were developed
for locally-convex (and, in particular, Banach) spaces
must be renounced. Indeed, $\lzer$ fails the local-convexity property
in a dramatic fashion: if $(\Omega,\FF,\PP)$ is non-atomic, the
topological dual of $\lzer$ is trivial (see \cite{KalPecRob84},
Theorem 2.2, p.~18).  Therefore, a new set of 
tools which do not rely on local convexity (and the related principles
such as the Hahn-Banach theorem) are needed to treat even the most
basic applied problems. Specifically, convexity has to be ``supplied
endogenously'', leading to various substitutes for indispensable
notions such as compactness (see \cite{DelSch99,Kom67,Zit09}). A
central idea behind their introduction is that a passage to a sequence
of convex combinations, instead of a more classical passage to a
subsequence, yields practically the same analytic benefit, while
working much better with the barren structure of $\lzer$. The situation is not as streamlined as in the classical case where
true subsequences are considered. Indeed, there are examples of
sequences $(f_n)_{\nin}$ in $\lzp$ (the nonnegative orthant of
$\lzer$) that converge to zero, whereas the set of all possible limits
of the convergent sequences $(h_n)_{\nin}$ such that
$h_n\in\conv(\set{f_n,f_{n+1},\dots})$ is the {\em entire $\lzp$} (see
Example \ref{exa:all-limit} for details). 

It is a goal of the present paper to give necessary and sufficient
conditions on a sequence $(f_n)_{\nin}$ in $\lzer_+$ to be {\em
  forward-convexly convergent}, i.e., such that each sequence of its
forward convex combinations (meaning a sequence $(h_n)_{\nin}$ with
$h_n\in\conv (\set{f_n,f_{n+1},\dots})$ for all $\nin$) converges in
$\lzer_+$ to the same limit. Arguably, forward-convex convergence
plays as natural a role in $\lzer$ as the strong convergence does in
$\lone$-spaces. It rules out certain pathological limits and, as will
be shown, imposes a measure-free locally-convex structure on the
sequence. Put simply, it brings the benefits of local convexity to a
naturally non-locally-convex framework.

As far as \emph{sufficient} conditions for forward-convex convergence
are concerned, the reader will quickly think of an example: almost
sure convergence of the original sequence will do, for instance.
Other than the obvious ones, useful \emph{necessary} conditions are
much harder to come by, and it is therefore surprising that one of our
main results has such a simple form. It says, \emph{inter alia}, that a
sequence $(f_n)_{\nin}$ is forward-convexly convergent {\em if and only if}
there exists a probability measure $\qprob$ in the equivalence class that generates the topology of $\lz$ such
that $(f_n)_{\nin}$ is $\Lb^1(\qprob)$-convergent. Effectively, it
identifies forward-convex convergence as a measure-free version of the
notion of uniform integrability.

Our main result also shows that failure of forward-convex convergence carries an interesting structure with it. In fact, when an $\lzp$-valued sequence that is $\lz$-convergent to $f \in \lzp$ fails to be forward-convexly convergent, the set $\C$ of all possible limits of its forward convex combinations is a \emph{strict} superset of $\set{f}$. A surprising amendment we add to
this statement is that $f \leq g$ holds in the almost sure sense for all $g\in\C$ --- in words, the limit $f$ is always the
smallest element in $\C$.  This extremality property can be viewed as
a measure-free no-loss-of-mass condition on the original sequence,
giving further support to the interpretation of forward-convex convergence as
a variant of uniform integrability.

\smallskip

After this introduction, we give a brief review of the notation and terminology and state
our main result in Section \ref{sec: results}. The proof of our main result is presented in Section \ref{sec: proof of main thm}.

\section{The Result} \label{sec: results}

\subsection{Preliminaries}

Let $(\Omega, \F, \oprob)$ be a probability space, and let $\Pi$ be
the collection of all probabilities on $(\Omega, \F)$ that are
equivalent to (the representative) $\oprob \in \Pi$.
All probabilities
in $\Pi$ have the same sets of zero measure, which we shall be calling
\textsl{$\Pi$-null}. We write ``$\Pi$-a.s.'' to mean $\prob$-a.s. with
respect to any, and then all, $\prob \in \Pi$.

By $\lp$ we shall be denoting the set of all (equivalence classes
modulo $\Pi$ of) \emph{possibly infinite-valued} nonnegative random
variables on $(\Omega, \F)$. We follow the usual practice of not
differentiating between a random variable and the equivalence class it
generates in $\lp$. The expectation of $f \in \lp$ under $\prob \in
\Pi$ is denoted by $\expecp [f]$. For fixed $\prob \in \Pi$, we define
a metric $d_\prob$ on $\lp$ via $d_\prob(f, g) = \expecp \bra{
  \abs{\exp(-f) - \exp(-g)}}$ for $f \in \lp$ and $g \in \lp$. The
topology on $\lp$ that is induced by the previous metric does not
depend on $\prob \in \Pi$; convergence of sequences in this topology
is simply (extended) convergence in probability under any $\prob \in
\Pi$.

A set $\C \subseteq \lp$ is \textsl{convex} if $\pare{\alpha f + (1 -
  \alpha) h} \in \C$ whenever $f \in \C$, $g \in \C$ and $\alpha \in
[0,1]$, where  the multiplication convention $0 \times
\infty = 0$ is used. For $\A \subseteq \lp$, $\conv(\A)$ denotes the smallest
convex set that contains $\A$; $\conv(\A)$ is just the set of all
possible finite convex combinations of elements in $\A$. Further,
$\oconv(\A)$ will denote the $\lp$-closure of $\conv(\A)$.

The set of all $f \in \lp$ such that $\set{f =
  \infty}$ is $\Pi$-null is denoted by $\lzp$. We endow $\lzp$ with
the restriction of the $\lp$-topology; convergence of sequences under
this topology is simply convergence in probability under any $\prob
\in \Pi$. When we write $\plimn f_n = f$, we tacitly imply that both
the sequence $(f_n)_{\nin}$ and the limit $f$ are elements of $\lzp$.

\subsection{Forward-convex convergence}
The following the a central notion of the paper.

\begin{defn}
Let $(f_n)_{\nin}$ be a sequence in $\lp$. Any sequence $(h_n)_{\nin}$ with the property that $h_n \in \conv(\set{f_n, f_{n+1}, \ldots })$ for all $\nin$ will be called a \textsl{sequence of forward convex combinations of $(f_n)_{\nin}$}.
\end{defn}

Since $\lzp$ is not a locally convex space, $\lzp$-convergence  of a sequence $(f_n)_{\nin}$ does not imply that sequences of forward convex combinations $(f_n)_{\nin}$ $\lzp$-converge to the same limit (or, for that matter, to any limit at all). We give an example of a quite pathological behavior.
\begin{exa}
\label{exa:all-limit}
  Take $\Omega = (0, 1]$ equipped with the Borel $\sigma$-field and
Lebesgue measure $\prob$, and define the sequence $(f_n)_{\nin}$  by 
\[
f_{n} = (m -1) 2^{m - 1} \indic_{((k-1) / 2^{m - 1}, \, k
/ 2^{m - 1}]}, \text{ for }n=2^{m-1} + k-1 \text{ with } m \in \Natural \text{ and } 1 \leq k \leq
2^{m-1}.
\]
It is straightforward to check that $\plimn f_n = 0$, but as we shall
show below,  this sequence behaves in a strange way: for any $f \in
\lzp$, there exists a sequence $(h_n)_{\nin}$ of forward convex
combinations of $(f_n)_{\nin}$ such that $\plimn h_n = f$.

We start by noting that it suffices to establish the above claim only for $f \in \lip$;
and, consequently, pick $f\in\lip$ with  $f \leq M$ for some $M \in \Real_+$. For each $m \in
\Natural$, let $\F_m$ be the $\sigma$-field on $\Omega$ generated by
the intervals $((k-1)2^{-m}, \, k 2^{-m}]$, $1 \leq k \leq 2^{m}$.
For $m \in \Natural$, define $g_m \dfn \expecp[f \such \F_m]$; by the martingale convergence theorem,  $\plim_{m \to
\infty} g_m = f$. Furthermore,
\[
g_m = \sum_{k =1}^{2^m} 2^m \expecp \bra{f \indic_{((k-1)/2^{m}, \, k
/2^{m}]} } \indic_{((k-1) / 2^{m}, \, k / 2^{m}]} = \sum_{k =1}^{2^m}
\frac{\expecp \bra{f \indic_{((k-1)/2^{m}, \, k /2^{m}]} }}{m} f_{2^m
+ k - 1}.
\]
Set $\alpha_{m, k} = m^{-1} \expecp \bra{f \indic_{((k-1)/2^{m}, \, k
/2^{m}]} } \in \Real_+$ for $m \in \Natural$ and $1 \leq k \leq 2^m$, so that, 
for $m \geq M$, we have
\[
\sum_{k=1}^{2^m} \alpha_{m, k} = \frac{\expecp[f]}{m} \leq \frac{M}{m} \leq 1.
\]
Define the sequence $(h_n)_{\nin}$ as follows: for $m \in \Natural$ with $m < M$, simply set
$h_{2^{m-1} + k-1} = f_{2^{m-1} + k-1}$ for all $1 \leq k \leq 2^{m-1}$,
while for $m \in \Natural$ with $m \geq M$ set
\[
h_{2^{m-1} + k-1} = \pare{1 - \sum_{\ell=1}^{2^m} \alpha_{m, \ell}}
f_{2^{m}} + \sum_{k =1}^{2^m} \alpha_{m, \ell} f_{2^m + \ell - 1} =
\pare{1 - \frac{\expecp[f]}{m}}
f_{2^{m}} + g_m
\]
for all $1 \leq k \leq 2^{m-1}$. Then, $(h_n)_{\nin}$ is a sequence of
forward convex combinations of $(f_n)_{\nin},$ and $\plimn h_n = f$.
\end{exa}

In the above example, note that the limit of $(f_n)_{\nin}$ is clearly minimal (in the $\Pi$-a.s. sense) in the set of all possible limits of sequences of forward convex combinations of $(f_n)_{\nin}$. As Theorem \ref{thm: main} will reveal, this did not happen by chance.

\subsection{The main result}

Having introduced all the ingredients, we are ready to state our main equivalence result.

\begin{thm} \label{thm: main} Let $(f_n)_{n \in \Natural}$ be a
  sequence in $\lz_+$. Assume that
  \begin{equation} \label{eq: CONV} \tag{CONV} \plimn f_n = f
  \end{equation}
  holds for some $f \in \lzp$. Then, the following statements are equivalent:
  \begin{enumerate}

  \item Every sequence of forward convex combinations of $(f_n)_{n \in
      \Natural}$ $\lzp$-converges to $f$.

  \item Whenever a sequence of forward convex combinations of
    $(f_n)_{n \in \Natural}$ is $\lp$-convergent, its $\lp$-limit is
     $f$.

  \item There exists $\qprob \in \Pi$ such that $\sup_{\nin}
    \expecq[f_n] < \infty$ and $\limn \expecq \big[ |f_n - f| \big] =
    0$.
  \end{enumerate}

\begin{itemize}
\item  With \eqref{eq: CONV} holding, and under any of the above equivalent
  conditions, we have
  \begin{equation} \label{eq: C_1 is simplex} \oconv (\set{f_1, f_2,
      \ldots}) = \set{\sum_{\nin} \alpha_n f_n + \pare{1 - \sum_{\nin}
        \alpha_n} f \ \bigg| \ (\alpha_n)_{\nin} \in \simplex}.
  \end{equation}
  where $\simplex$ is the infinite-dimensional simplex:
  \[
  \simplex \dfn \set{\alpha = (\alpha_n)_{\nin} \ \bigg| \ \alpha_n
    \in \Real_+ \text{ for all } \nin, \text{ and } \sum_{\nin}
    \alpha_n \leq 1}.
  \]
  Furthermore, with any $\qprob \in \Pi$ satisfying condition (3) above, $\oconv (\set{f_1, f_2, \ldots})$ is $\lonepq$-compact and the $\lzp$-topology on $\oconv
  (\set{f_1, f_2, \ldots})$ coincides with the $\lonepq$-topology. (In particular, $\oconv (\set{f_1, f_2, \ldots})$ with the
  $\lzp$-topology is locally convex and compact.)

\item With \eqref{eq: CONV} holding, if any of the equivalent conditions above fail, the set $\C \subseteq \lp$ of all possible $\lp$-limits of forward convex combinations of $(f_n)_{\nin}$ is such that $\set{f} \varsubsetneq \C$, and $f$ is minimal in $\C$ in the sense that $f \leq g$ holds $\Pi$-a.s. for all $g \in \C$.
\end{itemize}

  In the special case $f = 0$, the equivalences of the above three
  statements and the properties discussed after them hold even
  \emph{without} assumption \eqref{eq: CONV}.
\end{thm}

Implications $(1) \Rightarrow (2)$ and $(3) \Rightarrow (1)$ are straightforward, and \eqref{eq: CONV} is \emph{not} required. Indeed, $(1) \Rightarrow
(2)$ is completely trivial. Also, implication $(3)
\Rightarrow (1)$ is immediate since
\[
\limsupn \expecq \bra{\abs{h_n - f}} \leq
\limsupn \pare{\sup_{\Natural \ni k \geq n} \expecq \bra{|f_k - f|} }
= 0
\]
holds for any sequence $(h_n)_{n \in \Natural}$ of forward convex
combinations of $(f_n)_{\nin}$. The proof of implication $(2) \Rightarrow (3)$ is significantly
harder, and will be discussed in Section \ref{sec: proof of main thm}.

\begin{rem}
  Consider an $\lzp$-convergent sequence $(f_n)_{n \in \Natural}$, and
  set $f \dfn \plimn f_n$. From a qualitative viewpoint, Theorem
  \ref{thm: main} aids our understanding of the cases where a sequence
  $(h_n)_{\nin}$ of forward convex combinations of $(f_n)_{\nin}$
  $\lzp$-converges to a limit other than $f$. Indeed, in those cases
  $f$ is ``suboptimal'' in a very strong sense: all other possible
  limits of sequences of forward convex combinations of $(f_n)_{\nin}$
dominate it in the $\Pi$-a.s. pointwise sense.  
\end{rem}

\begin{rem}
  In the special case $f = 0$, \eqref{eq: CONV} is not needed in
  Theorem \ref{thm: main}. However, when $f \neq 0$, \eqref{eq: CONV}
  is crucial for $(2) \Rightarrow (1)$ of Theorem \ref{thm: main} to
  hold. We present an example to illustrate this fact. Assume that $(\Omega, \F, \prob)$ is rich enough to accommodate a
  sequence $(f_n)_{\nin}$ of random variables that are independent
  under $\prob$ and have identical distributions given by $\prob[f_n =
  0] = \prob[f_n = 2] = 1/2$. By Kolmogorov's zero-one law, it follows
  that any possible $\lp$-limit of sequences of convex combinations of
  $(f_n)_{\nin}$ has to be constant. Now, $(f_n)_{\nin}$ is uniformly
  integrable (in fact, uniformly bounded) under $\prob$, which means that the set $\C$ of all possible $\lp$-limit of sequence of convex combinations of $(f_n)_{\nin}$ is $\C = \{ 1 \}$. With $f = 1$ we have $(2)$ of Theorem \ref{thm: main}
  holding. However, both $(1)$ and $(3)$ fail.
\end{rem}

\section{Proof of Theorem \ref{thm: main}} \label{sec: proof of main thm}

\subsection{Preparatory remarks}

We start by mentioning a result \cite[Lemma A1.1]{MR1304434}, which will be used in a few places throughout the proof of Theorem \ref{thm: main}. Recall that a set $\B \subseteq \lz_+$ is called \pbdd \ if $\downarrow \lim_{\ell
  \to \infty} \sup_{f \in \B} \prob[f > \ell] = 0$ holds for some (and
then for all) $\prob \in \Pi$. If $\B \subseteq \lz_+$ is \pbdd, its
$\lp$-closure is a subset of $\lzp$, and coincides with its
$\lzp$-closure.

\begin{lem} \label{lem: A1.1}
Let $(g_n)_{\nin}$ be an $\lzp$-valued sequence. Then, there exists $h \in \lp$ and a sequence $(h_n)_{\nin}$ of forward convex combinations of $(g_n)_{\nin}$ such that $\lp$-$\limn h_n = h$. If, furthermore, $\conv \set{g_n \such \nin}$ is \pbdd, then $h \in \lzp$.
\end{lem}

We introduce some notation that will be used throughout the proof: for
$\nin$, set $\C_n\dfn \oconv \pare{\set{f_n, f_{n+1}, \ldots }} \subseteq
\lp$ so that $\C = \bigcap_{n \in \Natural} \C_n$. Also, let $\Sl_n
\subseteq \lp$ be the \emph{solid hull} of $\C_n$: $g \in \Sl_n$ if
and only if $0 \leq g \leq h$ for some $h \in \C_n$. It is clear that
$\Sl_n$ is convex and solid, and that $\C_n \subseteq \Sl_n$. Furthermore, set $\C \dfn \bigcap_{\nin} \C_n \subseteq \lp$. It is clear that $\C$ is the set of all possible $\lp$-limits of sequences of forward convex combinations of $(f_n)_{\nin}$. In particular, condition (2) of Theorem \ref{thm: main} can be succinctly written as $\C = \set{f}$.

We shall split the proof in several steps, indicating each time what is being
proved or discussed. Until the end of subsection \ref{subsec: proof in case f=0}, condition \eqref{eq: CONV} is not assumed.

\subsection{$\C \subseteq \lzp$ implies that $\oconv \pare{\set{f_1, f_2, \ldots }}$ is \pbdd} \label{subsec: proof that C1 is bdd}

We start by showing that $\Sl_n$ is $\lp$-closed, for $\nin$. For that, we pick an $\Sl_n$-valued sequence $(g_k)_{k \in \Natural}$ that converges
$\prob$-a.s.~to $g \in \lp$. Let $(h_k)_{k \in \Natural}$ be a
$\C_n$-valued sequence with $g_k \leq h_k$ for all $k \in
\Natural$. By Lemma \ref{lem: A1.1}, we can extract a sequence
$(\hti_k)_{k \in \Natural}$ of forward convex combinations of
$(h_k)_{k \in \Natural}$ such that $h \dfn \limk \hti_k \in \lp$
$\Pi$-a.s.~exists. Of course, $h \in \C_n$ and it is straightforward
that $g \leq h$. We conclude that $g \in \Sl_n$, i.e., $\Sl_n$ is
$\lp$-closed.

Let $\Sl = \bigcap_{\nin} \Sl_n$; then, $\C \subseteq \Sl$ and $\Sl$
is $\lp$-closed, convex and solid. We claim that $\Sl$ actually is the
solid hull of $\C$; to show this, we only need to establish that for
any $g \in \Sl$ there exists $h \in \C$ with $g \leq h$. For all
$\nin$, since $g \in \Sl \subseteq \Sl_n$, there exists $h_n \in \C_n$
with $g \leq h_n$. By another application of Lemma \ref{lem: A1.1}, we can extract a sequence $(\hti_n)_{n \in
  \Natural}$ of forward convex combinations of $(h_n)_{n \in
  \Natural}$ such that $h \dfn \lp \text{-} \limk \hti_k$
exists. Then, $h \in \C$ and $g \leq h$.

Each $\Sl_n$ is $\lp$-closed, convex and solid; therefore, a
straightforward generalization of \cite[Lemma 2.3]{MR1768009} gives,
for each $n \in \Natural$, the existence of a partition $\Omega =
\Phi_n \cup (\Omega \setminus \Phi_n)$, where $\Phi_n \in \F$, $\set{f
  \indic_{\Phi_n} \such f \in \Sl_n}$ is $\lzp$-bounded, while $h
\indic_{\Omega \setminus \Phi_n} \in \Sl_n$ for all $h \in \lp$.
% Note that $\C_{n} = \set{ \alpha f_{n} + (1 - \alpha) f \such \alpha
%   \in [0,1], \, f \in \C_{n + 1}}$ holds for each $n \in \Natural$.
Clearly, $\C_{n} \supseteq \C_{n+1}$ implies $\Phi_n \subseteq
\Phi_{n+1}$, for all $\nin$. However, since $f_n \in \lzp$, i.e.,
$\set{f_n = \infty}$ is $\Pi$-null for all $\nin$, it follows that
$\Phi_{n+1} = \Phi_{n}$ for all $\nin$. In other words, $\Phi_n =
\Phi_1$ for all $n \in \Natural$. Then, $h \indic_{\Omega \setminus
  \Phi_1} \in \Sl$ for all $h \in \lp$. Since $\C \subseteq \lzp$,
and, therefore, $\Sl \subseteq \lzp$ as well, it follows that $\Omega
\setminus \Phi_1$ is $\Pi$-null. Therefore, $\Sl_1$ is $\lzp$-bounded,
which completes this part of the proof. Observe that all $\Sl_n$,
$\nin$, are convex, solid, \pbdd, and $\lzp$-closed; we shall use this
later.

\subsection{Equivalence of $(1)$, $(2)$ and $(3)$ in Theorem \ref{thm: main} when $f = 0$} \label{subsec: proof in case f=0}

As already discussed, the proofs of $(1) \Rightarrow (2)$ and $(3) \Rightarrow (1)$ are
immediate, and \eqref{eq: CONV} is not used. Here, we  prove $(2)
\Rightarrow (3)$ when $f = 0$ without assuming \eqref{eq: CONV}.

Since $\Sl_1$ is convex, \pbdd \ and $\lzp$-closed, there exists $\prob \in \Pi$ such that $\sup_{h \in \Sl_1} \expecp[h] < \infty$. Although this result is somewhat folklore, we provide here a quick argument for its validity. Fixing a baseline probability $\oprob \in \Pi$, \cite[Theorem 1.1(4)]{Kar10a} implies that there exists $\hhat \in \Sl_1$ such that $\expec_{\oprob} [h / (1 + \hhat) ] \leq 1$ holds for all $h \in \Sl_1$. Define $\prob$ via the recipe $\ud \prob / \ud \oprob = c / \big( 1 + \hhat \big)$, where $c = 1 / \expec_{\oprob} [1 / (1 + \hhat)]$. Then, $\prob \in \Pi$ and
\[
\sup_{h \in \Sl_1} \expecp [h] = c \sup_{h \in \Sl_1} \expec_{\oprob} \bra{ \frac{h}{1 + \hhat}} \leq c < \infty.
\]
In particular, we have $\sup_{\nin} \expecp[f_n] < \infty$. Given the existence of such $\prob \in \Pi$, the following result will be useful in order to extract a probability
$\qprob \in \Pi$ that satisfies condition (3) of Theorem \ref{thm: main}.
\begin{lem} \label{lem: help with L-1 conv} Fix $\prob \in \Pi$ with
  $\sup_{\nin} \expecp[f_n] < \infty$. Then, the following statements
  are equivalent:
  \begin{enumerate}

  \item For some $\qprob \in \Pi$, $\sup_{\nin} \expecq[f_n] < \infty$
    and $\limn \expecq [f_n] = 0$.

  \item For any $\epsilon > 0$, there exists $A_\epsilon \in \F$ such
    that $\prob[\Omega \setminus A_\epsilon] \leq \epsilon$ and $\limn
    \expecp [f_n \indic_{A_\epsilon}] = 0$.
  \end{enumerate}
\end{lem}

\begin{proof}

First assume $(1)$ in the statement of Lemma \ref{lem: help with L-1 conv}. Define $Z \dfn \ud \qprob / \ud \prob$; then, $\prob[Z > 0] = 1$. For fixed $\epsilon > 0$, let $\delta = \delta (\epsilon) > 0$ be such that, with $A_\epsilon \dfn \set{Z > \delta} \in \F$,
  $\prob[\Omega \setminus A_\epsilon] \leq \epsilon$ holds. Then,
  \[
  \limsupn \expecp[f_n \indic_{A_\epsilon}] = \limsupn \expecq \big[
  (1/Z) f_n \indic_{\set{Z > \delta}} \big] \leq (1 / \delta) \limsup
  \expecq[f_n] = 0.
  \]

Now, assume $(2)$ in the statement of Lemma \ref{lem: help with L-1 conv}. For each $\kin$, let $B_k \in \F$ be such that
  $\prob[\Omega \setminus B_k] \leq 1/k$ and $\limn \expecp[f_n
  \indic_{B_k}] = 0$. By replacing $B_k$ with $\bigcup_{m=1}^k B_m$
  for each $k \in \Natural$ consecutively, we may assume without loss
  of generality that $(B_k)_{\kin}$ is a nondecreasing sequence of
  sets in $\F$ with $\limk \prob[B_k] = 1$, as well as that $\limn
  \expecp[f_n \indic_{B_k}] = 0$ holds for for each fixed
  $\kin$. Define $B_0 = \emptyset$, $n_0 = 0$, and a \emph{strictly}
  increasing $\Natural$-valued sequence $(n_k)_{\kin}$ with the
  following property: for all $\kin$, $\expecp[f_n \indic_{B_k}] \leq
  1 / k$ holds for all $n \geq n_{k - 1}$. (Observe that this is
  trivially valid for $k = 1$.) Then, define a sequence $(E_n)_{\nin}$
  of sets in $\F$ by setting $E_n = B_k$ whenever $n_{k-1} \leq n <
  n_{k}$. It is clear that $(E_n)_{\nin}$ is a nondecreasing sequence,
  that $\limn \prob[E_n] = 1$, and that $\limn \expecp[f_n
  \indic_{E_n}] = 0$. With $E_0 \dfn \emptyset$, define $Z \dfn c
  \sum_{n \in \Natural} 2^{-n} \indic_{E_n \setminus E_{n-1}}$, where
  $c > 0$ is a normalizing constant in order to ensure that
  $\expecp[Z]= 1$. Define $\qprob \in \Pi$ via $\qprob[A] = \expecp[Z
  \indic_A]$ for all $A \in \F$. With $K \dfn \sup_{\nin} \expecp[f_n]
  < \infty$, $\sup_{\nin} \expecq[f_n] \leq c \sup_{\nin} \expecp[f_n]
  = c K < \infty$. Furthermore,
  \[
  \expecq[f_n] = \expecq[f_n \indic_{E_n}] + \expecq[f_n
  \indic_{\Omega \setminus E_n}] \leq c \expecp[f_n \indic_{E_n}] + c
  2^{-n} \expecp[f_n \indic_{\Omega \setminus E_n}] \leq c \expecp[f_n
  \indic_{E_n}] + c K 2^{-n}.
  \]
  Since $\limn \expecp[f_n \indic_{E_n}] = 0$, we obtain $\limn
  \expecq[f_n] = 0$, which completes the argument.
\end{proof}

We continue with the proof of the implication  $(2) \Rightarrow (3)$, \emph{fixing $\prob
  \in \Pi$ with $\sup_{\nin} \expecp [f_n] < \infty$ until the end of
  \S \ref{subsec: proof in case f=0}}.

For any $\A \subseteq \lzp$, define its \textsl{polar} $\A^\circ \dfn
\set{g \in \lzp \such \expecp[g h] \leq 1 \text{ for all } h \in
  \A}$. It is straightforward that $\pare{\bigcup_{\nin} \A_n}^\circ =
\bigcap_{\nin} \A^\circ_n$, for all collections $\set{\A_n \such
  \nin}$ of subsets of $\lzp$. Also, consider the \textsl{bipolar}
$\A^{\circ \circ} \dfn (\A^\circ)^\circ$ of $\A$; Theorem 1.3 of
\cite{MR1768009} states that if a set is convex and solid, $\A^{\circ
  \circ}$ coincides with the $\lzp$-closure of $\A$.

For each $\nin$, $\Sl_n \subseteq \lzp$ is convex, solid and
$\lzp$-closed; therefore, $\Sl_n^{\circ \circ} = \Sl_n$. Since $\Sl =
\bigcap_{\nin}\Sl_n$ is the solid hull of $\C = \set{0}$, i.e., $\Sl =
\set{0}$, we have
\[
\pare{\bigcup_{\nin}\Slo_n}^{\circ \circ}
= \pare{\bigcap_{\nin}\Sl_n^{\circ \circ}}^\circ
= \pare{\bigcap_{\nin}\Sl_n}^\circ = \set{0}^\circ = \lzp.
\]
Since $\bigcup_{\nin}\Slo_n$ is convex and solid, the above means that
the $\lzp$-closure of $\bigcup_{\nin}\Slo_n$ is $\lzp$.

Fix $\epsilon > 0$. Define a $\Natural$-valued and strictly increasing
sequence $(n_k)_{\kin}$ with the following property: for all $\kin$
there exists $g_k \in \Slo_{n_k}$ such that $\prob[|g_k - 2 k| \leq k]
\leq \epsilon 2^{-(k+1)}$. (This can be done in view of the fact that
the $\lzp$-closure of $\bigcup_{\nin}\Slo_n$ is $\lzp$.) In
particular, $\prob[g_k \leq k] \leq \epsilon 2^{-(k+1)}$ and
$\expecp[g_k f_n] \leq 1$ hold for all $\kin$ and $n \geq n_k$. Define
$A_\epsilon \dfn \bigcap_{\kin} \set{g_k > k}$; then, $\prob[\Omega
\setminus A_\epsilon] \leq \epsilon$. Furthermore, for all $k \in
\Natural$ and $n \geq n_k$,
\[
\expecp[f_n \indic_{A_\epsilon}] \leq \expecp[f_n \indic_{\{ g_k > k
  \}} ] \leq \expecp[ (g_k / k) f_n \indic_{\{ g_k > k \}}] \leq (1 /
k) \expecp[ g_k f_n ] \leq 1 / k.
\]
Then, $\limn \expecp[f_n \indic_{A_\epsilon}] = 0$. Invoking Lemma \ref{lem: help with L-1 conv}, we obtain the existence of
$\qprob \in \Pi$ such that $\sup_{\nin} \expecq[f_n] < \infty$ and
$\limn \expecq [f_n] = 0$.

\subsection{A domination result}
The next simple result will be important for the development.

\begin{prop} \label{prop: domination}
Let $(f_n)_{\nin}$ satisfy \eqref{eq: CONV}. Furthermore, let $(g_n)_{\nin}$ be a sequence of forward convex combinations of $(f_n)_{\nin}$ such that $\lp$-$\limn g_n = g$. Then, $\Pi$-a.s., $f \leq g$.
\end{prop}

\begin{proof}
By way of contradiction, assume that $\oprob[f > g] > 0$, where $\oprob \in \Pi$, and consider the probability $\prob$ which is $\oprob$ conditioned on the event $\{ f > g\}$. Note that $\prob[f > g] = 1$ and that $\limn f_n = f$ and $\limn g_n = g$ still hold under the measure $\prob$. Let $U: [0, \infty] \mapsto [0,1]$ be the strictly increasing and concave function defined via $U(x) = 1 - \exp(-x)$ for $x \in [0, \infty]$. The dominated convergence theorem implies that $\limn \expecp \bra{U(f_n)} = \expecp \bra{U(f)}$ and $\limn \expecp \bra{U(g_n)} = \expecp \bra{U(g)}$. In view of the concavity of $U$, one has $\expecp \bra{U(g_n)} \geq \inf_{k \geq n} \expecp \bra{U(f_k)}$ for all $\nin$. Since $\limn \inf_{k \geq n} \expecp \bra{U(f_k)} = \expecp \bra{U(f)}$, we conclude that $\expecp\bra{U(f)} \leq \expecp \bra{U(g)}$. The last inequality combined with the fact that $U$ is strictly increasing contradicts $\prob[f > g] = 1$. Therefore, we conclude that, $\Pi$-a.s., $f \leq g$.
\end{proof}

\subsection{Equivalence of $(1)$, $(2)$ and $(3)$ in Theorem \ref{thm: main}: general case} \label{subsec: proof - general case}

We shall now tackle the general case $f \in \lzp$, working under the assumption \eqref{eq: CONV}. Of course $(1) \Rightarrow (2)$ and $(3) \Rightarrow (1)$ are still trivially valid.  The proof of $(2) \Rightarrow (3)$ will be reduced to the special case $f = 0$, which we have already established, via the following result.

\begin{lem} \label{lem: forw conv combos equiv} The following
  statements are equivalent:

  \begin{enumerate}

  \item[(i)] Every sequence of forward convex combinations of
    $(f_n)_{\nin}$ $\lzp$-converges to $f$.

  \item[(ii)] Every sequence of forward convex combinations of $(|f_n -
    f|)_{\nin}$ $\lzp$-converges to zero.

  \end{enumerate}
\end{lem}

\begin{proof}
As $(ii) \Rightarrow (i)$ is immediate, we only treat implication $(i) \Rightarrow (ii)$. For $x \in \Real$ and $y \in \Real$, we denote by $x_+$ the positive part of $x$ and by $x \wedge y$ the minimum between $x$ and $y$. We shall first argue that every sequence of forward convex combinations of $\pare{(f - f_n)_+}_{\nin}$ $\lzp$-converges to zero. The facts that $0 \leq f_n \wedge f \leq f$ for all $\nin$ and $\lzp$-$\limn f_n = f$, coupled with Proposition \ref{prop: domination}, imply that whenever a sequence of forward convex combinations of $(f_n \wedge f)_{\nin}$ $\lp$-converges, the limit is $f$. Since $(f - f_n)_+ = f - f \wedge f_n$, it follows that whenever a sequence of forward convex combinations of $\pare{(f - f_n)_+}_{\nin}$ $\lp$-converges, the limit is zero. From the special case of Theorem \ref{thm: main} that we have established previously, it actually follows that every sequence of forward convex combinations of $\pare{(f - f_n)_+}_{\nin}$ $\lp$-converges to zero. Note also that this implies that every sequence of forward convex combinations of $\pare{(f \wedge f_n)_+}_{\nin}$ $\lp$-converges to $f$. We now proceed in showing that every sequence of forward convex combinations of $\pare{(f_n - f)_+}_{\nin}$ $\lp$-converges to zero, which will complete the argument. The fact that $(f_n - f)_+ = f_n - f \wedge f_n$ for all $\nin$ and Proposition \ref{prop: domination} imply that whenever a sequence of forward convex combinations of $\pare{(f_n - f)_+}_{\nin}$ $\lp$-converges, its limit is zero. Once again, by the special case of Theorem \ref{thm: main} that we have established previously, it actually follows that every sequence of forward convex combinations of $\pare{(f_n - f)_+}_{\nin}$ $\lp$-converges to zero. This completes the proof.
\end{proof}

In view of the result of Lemma \ref{lem: forw conv combos equiv} and
the treatment in subsection \ref{subsec: proof in case f=0}, we obtain the
existence of $\qprob \in \Pi$ such that $\sup_{\nin} \expecq [|f_n -
f|] < \infty$ and $\limn \expecq [|f_n - f|] = 0$. Replacing $\qprob$,
if necessary, by $\qprob' \in \Pi$ defined via $\ud \qprob' / \ud
\prob = c (1 + f)^{-1}$ where $c = \pare{\expecq [(1 +
  f)^{-1}]}^{-1}$, we may further assume that $\expecq [f] < \infty$;
in other words, $\sup_{\nin} \expecq [f_n] < \infty$ and $\limn
\expecq [|f_n - f|] = 0$.

\subsection{Proof of claims after the equivalences}

To begin with, assume that any of the equivalent statements of Theorem \ref{thm: main} is not valid. Then, we must have that $\set{f} \subsetneq \C$, since $\C = \set{f}$ is actually statement (2). Then, Proposition \ref{prop: domination} implies that for all $g \in \C$ we have, $\Pi$-a.s., $f \leq g$.

Continuing, assume the validity of any of the equivalent statements of Theorem \ref{thm: main}. The following result will help to establish all the properties of $\C_1$ that are mentioned in Theorem \ref{thm: main}.

\begin{lem} \label{lem: C1 is compact} Let $\C'_1\subseteq \lp$ be the
  set on the right-hand-side of \eqref{eq: C_1 is simplex}. If $\qprob
  \in \Pi$ is such that condition (5) of Theorem \ref{thm: main}
  holds, then $\C'_1$ is $\Lb^1_+ (\qprob)$-compact.
\end{lem}

\begin{proof}
  First of all, since $\sup_{\nin} \expecq [f_n] < \infty$, which in
  particular implies that $\expecq[f] < \infty$ by Fatou's lemma, it
  is clear that $\sup_{g \in \C'_1} \expecq[g] < \infty$ --- in
  particular, $\C'_1 \subseteq \lzp$.

  We shall show that \emph{any} sequence $(g_k)_{\kin}$ in $\C'_1$ has
  an $\Lb^1_+ (\qprob)$-convergent subsequence. For all $\kin$, write
  $g_k = \sum_{\nin} \alpha_{k, n} f_n + (1 - \sum_{\nin} \alpha_{k,
    n}) f$, where $\alpha_k = (\alpha_{k,n})_{\nin} \in \simplex$. By
  a diagonalization argument, we can find a subsequence of
  $(g_k)_{\kin}$, which we shall still denote by $(g_k)_{\kin}$, such
  that $\alpha_n \dfn \limk \alpha_{k, n}$ exists for all
  $\nin$. Fatou's lemma implies that $\alpha = (\alpha_n)_{\nin} \in
  \simplex$. Let $g \dfn \sum_{\nin} \alpha_{n} f_n + (1 - \sum_{\nin}
  \alpha_{n}) f$. We shall show that $\limk \expecq [|g_k - g|] =
  0$. For $\epsilon > 0$, pick $N = N (\epsilon) \in \Natural$ such
  that $\sup_{\nin} \expecq[|f_{N+n} - f|] \leq \epsilon /2$. Define
  $g^{(N)} \dfn \sum_{n=1}^N \alpha_{n} f_n + (1 - \sum_{n=1}^N
  \alpha_{n}) f$, as well as $g_k^{(N)} \dfn \sum_{n=1}^N \alpha_{k,
    n} f_n + (1 - \sum_{n=1}^N \alpha_{k, n}) f$ for all
  $\kin$. Observe that
  \[
  \expecq \bra{ \abs{g^{(N)} - g}} = \expecq \bra{ \abs{\sum_{\nin}
      \alpha_{N+n} (f_{N+n} - f)}} \leq \sum_{\nin} \alpha_{N + n }
  \expecq \bra{ \abs{ f_{N + n} - f}} \leq \frac{\epsilon}{2}
  \]
  Similarly, $\expecq \big[ \big| g_k^{(N)} - g_k \big| \big] \leq
  \epsilon / 2$ holds for all $\kin$. Furthermore,
  \[
  \limsup_{k \to \infty} \expecq \bra{ \abs{g_k^{(N)} - g^{(N)}}} \leq
  \limsup_{k \to \infty} \pare{\sum_{n=1}^N |\alpha_{k, n} - \alpha_n|
    \, \expecq [|f_n - f|]} = 0.
  \]
  It follows that $\limsup_{k \to \infty} \expecq \bra{ \abs{g_k - g}}
  \leq \epsilon$. Since $\epsilon > 0$ is arbitrary, $\lim_{k \to
    \infty} \expecq \bra{ \abs{g_k - g}} = 0$.
\end{proof}

To finish the proof of Theorem \ref{thm: main}, it remains to show
that $\C_1 = \C'_1$ and that the $\lzp$-topology coincides with the
$\Lb^1_+(\qprob)$-topology on $\C_1$. First of all, since $f \in
\C_1$, $f_n \in \C_1$ for all $\nin$, and $\C_1$ is closed, we have
$\C'_1 \subseteq \C_1$. On the other hand, $\conv(\{ f_1, f_2, \ldots
\}) \subseteq \C'_1$; since $\C'_1$ is $\lzp$-closed by Lemma
\ref{lem: C1 is compact}, $\C_1 = \oconv(\{ f_1, f_2, \ldots \})
\subseteq \C'_1$. Therefore, $\C_1 = \C'_1$. Finally, let
$(g_k)_{\kin}$ be a $\C_1$-valued and $\lzp$-convergent sequence, and
call $g \dfn \plimk g_k \in \C_1$. Lemma \ref{eq: C_1 is simplex}
implies that every subsequence of $(g_k)_{\kin}$ has a further
subsequence that is $\Lb^1_+(\qprob)$-convergent. All the latter
subsequences have to $\Lb^1_+(\qprob)$-converge to $g$, which means
that $(g_k)_{\kin}$ $\Lb^1_+(\qprob)$-converges to $g$.

% ----------------------------------------------------------------
\bibliographystyle{siam} 
\bibliography{fcc}
\end{document}